\documentclass{article}
\usepackage[T1]{fontenc}
\usepackage{amsmath,amsfonts,amssymb,amsthm, color,wrapfig,float,dsfont}
\usepackage{authblk}
\usepackage[french,english]{babel}

\title{Granular media equation with double-well external landscape: limiting steady state}
\author[*]{Julian Tugaut}
\affil[*]{Univ Lyon, Universit\'e Jean Monnet, CNRS UMR 5208, Institut Camille Jordan, Maison de l'Universit\'e, 10 rue Tr\'efilerie, CS 82301, 42023 Saint-\'Etienne Cedex 2, France}

\begin{document}

\newcommand{\bRb}{\mathbb{R}}
\newcommand{\bCb}{\mathbb{C}}
\newcommand{\bEb}{\mathbb{E}}
\newcommand{\bKb}{\mathbb{K}}
\newcommand{\bQb}{\mathbb{Q}}
\newcommand{\bFb}{\mathbb{F}}
\newcommand{\bGb}{\mathbb{G}}
\newcommand{\bNb}{\mathbb{N}}
\newcommand{\bZb}{\mathbb{Z}}

\newcommand{\deriv}{\stackrel{\mbox{\bf\Large{}$\cdot$\normalsize{}}}}
\newcommand{\dederiv}{\stackrel{\mbox{\bf\Large{}$\cdot\cdot$\normalsize{}}}}

\theoremstyle{plain} \newtheorem{thm}{Theorem}[section]
\theoremstyle{plain} \newtheorem{prop}[thm]{Proposition}
\theoremstyle{plain} \newtheorem{props}[thm]{Properties}
\theoremstyle{plain} \newtheorem{ex}[thm]{Example}
\theoremstyle{plain} \newtheorem{contrex}[thm]{Coounterexample}
\theoremstyle{plain} \newtheorem{cor}[thm]{Corollary}
\theoremstyle{plain} \newtheorem{hyp}[thm]{Hypothesis}
\theoremstyle{plain} \newtheorem{assu}[thm]{Assumption}
\theoremstyle{plain} \newtheorem{hyps}[thm]{Hypotheses}
\theoremstyle{plain} \newtheorem{lem}[thm]{Lemma}
\theoremstyle{plain} \newtheorem{rem}[thm]{Remark}
\theoremstyle{plain} \newtheorem{nota}[thm]{Notation}
\theoremstyle{plain} \newtheorem{defn}[thm]{Definition}

\newcommand{\cRc}{\mathcal{R}}
\newcommand{\cCc}{\mathcal{C}}
\newcommand{\cEc}{\mathcal{E}}
\newcommand{\cKc}{\mathcal{K}}
\newcommand{\cQc}{\mathcal{Q}}
\newcommand{\cFc}{\mathcal{F}}
\newcommand{\cGc}{\mathcal{G}}
\newcommand{\cNc}{\mathcal{N}}
\newcommand{\cZc}{\mathcal{Z}}
\newcommand{\cOc}{\mathcal{O}}
\newcommand{\cSc}{\mathcal{S}}
\newcommand{\cAc}{\mathcal{A}}
\newcommand{\cBc}{\mathcal{B}}
\newcommand{\cIc}{\mathcal{I}}
\newcommand{\cDc}{\mathcal{D}}
\newcommand{\cLc}{\mathcal{L}}
\newcommand{\cHc}{\mathcal{H}}

\newcommand{\Ima}{{\rm Im}}
\newcommand{\Rea}{{\rm Re}}
\newcommand{\gaga}{\left|\left|}
\newcommand{\drdr}{\right|\right|}
\newcommand{\lra}{\left\langle}
\newcommand{\rra}{\right\rangle}

\newcommand{\EE}{\mathbb{E}}
\newcommand{\PP}{\mathbb{P}}

\newcommand{\sepa}{\left|\right.}

\newcommand{\crg}{[\![}
\newcommand{\crd}{]\!]}

\newcommand{\sgn}{{\rm Sign}}
\newcommand{\vari}{{\rm Var}}
\newcommand{\cov}{{\rm Cov}}
\newcommand{\poin}[1]{\dot{#1}}
\newcommand{\norm}[1]{\Vert #1\Vert}

\maketitle

\begin{abstract}
In this paper, we give a simple condition on the initial state of the granular media equation which ensures that the limit as the time goes to infinity is the unique steady state with positive center of mass. To do so, we use functional inequalities, Laplace method and McKean-Vlasov diffusion (which corresponds to the probabilistic interpretation of the granular media equation).
\end{abstract}
\medskip

 \par\medskip

\par\medskip

\section{Introduction}

In this work, we are interested in the long-time behavior of the following granular media equation in the one-dimensional setting:

\begin{equation}
\label{granular}
\frac{\partial}{\partial t}\mu_t^\sigma(x)=\frac{\sigma^2}{2}\frac{\partial^2}{\partial x^2}\mu_t^\sigma(x)+\frac{\partial}{\partial x}\left\{\mu_t^\sigma(x)\left(\nabla V(x)+\alpha(x-m_1^\sigma(t))\right)\right\}
\end{equation}

with $\alpha>0$ and $m_1^\sigma(t):=M\left(\mu_t^\sigma\right)$ where for any measure $\mu$, we put:

\begin{equation}
\label{Pibis2}
M(\mu):=\int_\bRb x\mu(dx)\,.
\end{equation}

Since $m_1^\sigma(t)$ depends on $\mu_t^\sigma$, Equation~\eqref{granular} is nonlinear. Let us point out that it may be rewritten like so:

\begin{equation}
\label{granular2}
\frac{\partial}{\partial t}\mu_t^\sigma(x)=\frac{\sigma^2}{2}\frac{\partial^2}{\partial x^2}\mu_t^\sigma(x)+\frac{\partial}{\partial x}\left\{\mu_t^\sigma(x)\left(\nabla V(x)+(\nabla F\ast\mu_t^\sigma)(x)\right)\right\}\,,
\end{equation}

with $F(x):=\frac{\alpha}{2}x^2$ and with $\ast$ standing for the convolution.

The precise assumptions on $V$ will be given subsequently.

This work deals with the fully nonconvex case: $V$ has two wells located in $a>0$ and in $-a$. In this setting, it is also well-known that if $\sigma$ is small enough, then there are several steady states, see \cite{HT1,PT}. Despite this non-uniqueness, the convergence towards one of the invariant probability measure has been proven in \cite{AOP}.

Nevertheless, very few is known about the basins of attraction. In \cite{Tamura1984}, it has been proven that if the initial measure $\mu_0^\sigma$ is sufficiently close to a steady state $\mu^\sigma$, then $\mu_t^\sigma$ converges towards $\mu^\sigma$ albeit there is no quantification at all of the expression ``sufficiently close''. Moreover, the temperature $\frac{\sigma^2}{2}$ is equal to one in the latter result. As a consequence, there is no uniformity in the previous result with respect to $\sigma$. In \cite{AOP}, some conditions have been established ensuring that the limiting probability measure has a positive expectation but there are difficult to check. Indeed, there are related to the initial free-energy and the infimum of this free-energy on an hyperplane of measures. Finally, in \cite{KRM}, some condition has been found for small temperature albeit there is no quantification at all of the expression ``small temperature''.

In this work, we assume that there are exactly three steady states with total mass equal to $1$: $\nu_-^\sigma$, $\nu_+^\sigma$ and $\nu_0^\sigma$ where $\int_\bRb x\nu_0^\sigma(dx)=0$ and $\pm\int_\bRb x\nu_\pm^\sigma(dx)>0$. We then exhibit a simple condition ensuring that $\mu_t^\sigma$ converges towards $\nu_+^\sigma$.

The main ingredient to establish the result is the self-stabilizing diffusion, that is a stochastic process which probability law is satisfying Equation~\eqref{granular}:

\begin{equation}
\label{MKV}
X_t^\sigma=X_0+\sigma B_t-\int_0^t\nabla V(X_s^\sigma)ds-\alpha\int_0^t\left(X_s^\sigma-\mathbb{E}\left[X_s^\sigma\right]\right)ds\,.
\end{equation}

Here, $X_0$ is a random variable, independent from the Brownian motion $B$ and we assume that $X_0$ follows the law $\mu_0^\sigma$.

This kind of processes have been introduced in the seminal works \cite{McKean1966,McKean}. Let us point out that $\mathcal{L}(X_t^\sigma)=\mu_t^\sigma$ for any $t\geq0$.

We will also write the equation satisfied by $X^\sigma$ in this way:

\begin{equation}
\label{MKV2}
X_t^\sigma=X_0+\sigma B_t-\int_0^t\nabla V(X_s^\sigma)ds-\alpha\int_0^t\left(X_s^\sigma-m_1^\sigma(s)\right)ds\,.
\end{equation}

We will not discuss the well-posedness of Equation~\eqref{MKV}. About this, we refer to \cite{HIP}. 

In \cite{PT}, we have proven the existence of $\sigma_c(\alpha)>0$ such that if $\sigma\geq\sigma_c(\alpha)$, there is a unique invariant probability measure and if $\sigma<\sigma_c(\alpha)$, there are exactly three. In this work, we assume that $\sigma<\sigma_c(\alpha)$ to avoid the obvious case.

After giving the precise assumptions on the potential~$V$ and on the initial measure $\mu_0^\sigma$, we provide a preliminary subsection. After, we state the main results that is Theorem~\ref{thm:A} and Corollary~\ref{cor:mum} (with its proof). In a last section, we give the proof of the theorem.

\subsection*{Assumptions}

In the whole paper, we assume that $V$ satisfies the following set of hypotheses, similar to the ones in \cite{PT}:

\noindent{}{\bf (V-1)} \emph{Polynomial function:} $V$ is a polynomial function with $\deg(V)\geq4$.\\
{\bf (V-2)} \emph{Symmetry:} $V$ is an even function.\\
{\bf (V-3)} \emph{Double-well potential:} The equation $V'(x)=0$ admits exactly three solutions: $a$, $-a$ and $0$ with $a>0$;  $V''(a)>0$ and $V''(0)<0$. The bottoms of the wells are reached for $x=a$ and $x=-a$. Moreover, $V^{(2k)}(0)\geq0$, for all $k\geq2$.\\
{\bf (V-4)} $\displaystyle\lim_{x\to\pm\infty}V''(x)=+\infty$ and for any $x\geq a$, $V''(x)>0$.\\
{\bf (V-5)} \emph{Initialization:} $V(0)=0$.

In the following, the nonconvexity of $V$ will be measured by the following constant: $\theta:=\sup_\bRb-V''$. We also assume, eventually, that synchronization occurs that means: $\alpha>\theta$.

\medskip

\noindent{}The initial measure $\mu_0^\sigma$ is assumed to be absolutely continuous with respect to the Lebesgue measure with a density that we denote for simplicity by $\mu_0^\sigma$. Moreover, $\int_\bRb x^{2k}\mu_0^\sigma(x)dx<\infty$ for any $k\in\mathbb{N}$ and $\int_\bRb\mu_0^\sigma(x)\log(\mu_0^\sigma(x))dx<+\infty$ that is the initial entropy is finite and so the same is true for the initial free-energy.

\subsection*{Preliminary}

In \cite{HT1,PT}, we have proven that if $\mu^\sigma$ is a steady state with total mass equal to $1$ of the granular media equation \eqref{granular} then there exists $m\in\bRb$ such that $\mu^\sigma=\mu^{m,\sigma}$ where the measure $\mu^{m,\sigma}$ is defined as

\begin{equation}
\label{Mum}
\mu^{m,\sigma}(dx):=\frac{\exp\left[-\frac{2}{\sigma^2}\left(V(x)+\frac{\alpha}{2}x^2-\alpha mx\right)\right]}{\int_\bRb\exp\left[-\frac{2}{\sigma^2}\left(V(y)+\frac{\alpha}{2}y^2-\alpha my\right)\right]dy}\,dx\,,
\end{equation}

Moreover, $m$ is a zero of the following function:

\begin{equation}
\label{Chisigma}
\chi_\sigma(m):=M\left(\mu^{m,\sigma}\right)-m=\frac{\int_\bRb x\exp\left[-\frac{2}{\sigma^2}\left(V(x)+\frac{\alpha}{2}x^2-\alpha mx\right)\right]dx}{\int_\bRb\exp\left[-\frac{2}{\sigma^2}\left(V(x)+\frac{\alpha}{2}x^2-\alpha mx\right)\right]dx}-m\,.
\end{equation}

In \cite{PT}, we have shown that if $\sigma<\sigma_c(\alpha)$, then the function $\chi_\sigma$ is increasing on an interval $[0;t_\sigma]$ then decreasing on $[t_\sigma;+\infty[$ where $t_\sigma>0$. Moreover, $\chi_\sigma(+\infty)=-\infty$ and $\chi_\sigma(0)=0$. Since the function $\chi_\sigma$ is odd, it admits exactly three zero: $0$, $m(\sigma)>0$ and $-m(\sigma)$.

As a consequence, there are three steady states with total mass equal to $1$: $\nu_0^\sigma:=\mu^{0,\sigma}$, $\nu_+^\sigma:=\mu^{m(\sigma),\sigma}$ and $\nu_-^\sigma:=\mu^{-m(\sigma),\sigma}$.

\subsection*{Main results}

We now give the main results about the choice of the limiting probability in the long-time behavior.

\begin{thm}
\label{thm:A}
We assume that there exists $\delta\in]0;m(\sigma)[$ such that 

\begin{equation}
\label{Condition1}
\int_\bRb x\mu_0^{\sigma}(dx)>m(\sigma)-\delta\,,
\end{equation}

and

\begin{equation}
\label{Condition2}
\chi_\sigma(m(\sigma)-\delta)>
\left\{
\begin{array}{ll}
	\mathbb{W}_2\left(\mu_0^{\sigma};\mu^{m(\sigma)-\delta,\sigma}\right)&\quad if\quad\alpha>\theta\\
	{\rm L}_2\left(\mu_0^{\sigma};\mu^{m(\sigma)-\delta,\sigma}\right)&\quad if\quad\alpha\leq\theta
\end{array}\right.\,,
\end{equation}

where $\mathbb{W}_2$ stands for the quadratic Wasserstein distance. Then, $\mu_t^\sigma$ converges weakly towards $\nu_+^\sigma$ as $t$ goes to infinity.
\end{thm}

As mentioned above, $\chi_\sigma(m)\leq0$ for any $m\geq m(\sigma)$ so we immediately deduce that the theorem can not be applied if $\int_\bRb x\mu_0^{\sigma}(dx)\geq m(\sigma)$.

We now give an immediate corollary:

\begin{cor}
\label{cor:mum}
We assume that there exists $m\in]0;m(\sigma)[$ such that $\mu_0^\sigma=\mu^{m,\sigma}$. Then, $\mu_t^\sigma$ converges weakly towards $\nu_+^\sigma$ as $t$ goes to infinity.
\end{cor}

\begin{proof}
Hypothesis \eqref{Condition2} is satisfied since for any $m\in]0;m(\sigma)[$, there exists $\delta\in]0;m(\sigma)[$ such that $m=m(\sigma)-\delta$. Then, we just remark: $\chi_\sigma(m(\sigma)-\delta)>0=\mathbb{W}_2\left(\mu^{m(\sigma)-\delta,\sigma};\mu^{m(\sigma)-\delta,\sigma}\right)={\rm L}_2\left(\mu^{m(\sigma)-\delta,\sigma};\mu^{m(\sigma)-\delta,\sigma}\right)$. Hypothesis \eqref{Condition1} is also satisfied due to $\int_\bRb x\mu^{m(\sigma)-\delta,\sigma}(dx)=M\left(\mu^{m(\sigma)-\delta,\sigma}\right)=m(\sigma)-\delta+\chi_\sigma(m(\sigma)-\delta)>m(\sigma)-\delta$. Thus, we may apply Theorem~\ref{thm:A}.
\end{proof}

We point out that we have a similar result with the other half-line: if $\mu_0^\sigma=\mu^{-m(\sigma)+\delta,\sigma}$. Then, $\mu_t^\sigma$ converges weakly towards $\nu_-^\sigma$ as $t$ goes to infinity.

\bigskip

{\bf Acknowledgements: }This work is supported by the  French ANR grant METANOLIN (ANR-19-CE40-0009).

\section{Proof of Theorem~\ref{thm:A}}

We introduce the following diffusion

\begin{equation}
\label{EDS:frozen2}
Y_t^\sigma=X_0+\sigma B_t-\int_0^t\nabla V(Y_s^\sigma)ds-\alpha\int_0^t\left(Y_s^\sigma-\left(m(\sigma)-\delta\right)\right)ds\,.
\end{equation}

We consider the \emph{deterministic} time $T_0^\sigma:=\inf\left\{t\geq0\,\,:\,\,\mathbb{E}(X_t^\sigma)\leq m(\sigma)-\delta\right\}$ where $X^\sigma$ is defined in Equation~\eqref{MKV}. Then, for any $t\in[0;T_0^\sigma]$, we have $X_t^\sigma\geq Y_t^\sigma$ so $\EE\left(X_t^\sigma\right)\geq\EE\left(Y_t^\sigma\right)$. However, $\EE\left(Y_t^\sigma\right)=\EE\left(Y_\infty^\sigma\right)+\EE\left(Y_t^\sigma-Y_\infty^\sigma\right)$ where $Y_\infty^\sigma$ follows the law $\mu^{m(\sigma)-\delta,\sigma}$, which is the unique invariant probability of Diffusion~\eqref{EDS:frozen2}.

In the one hand, we assume that $\alpha>\theta$. Then, the potential $x\mapsto V(x)+\frac{\alpha}{2}\left(x-(m(\sigma)-\delta)\right)^2$ is uniformly strictly convex. Also, the unique invariant probability measure of Equation~\eqref{EDS:frozen2} is the measure $\mu^{m(\sigma)-\delta,\sigma}$. Hence, by applying \cite[Theorem 2.1.]{BGG1}, we deduce:

\begin{equation*}
\mathbb{W}_2\left(\mathcal{L}\left(Y_t^\sigma\right);\mu^{m(\sigma)-\delta,\sigma}\right)\leq e^{-(\alpha-\theta)t}\mathbb{W}_2\left(\mathcal{L}\left(X_0\right);\mu^{m(\sigma)-\delta,\sigma}\right)\,.
\end{equation*}

As a consequence: 

\begin{align*}
\EE\left(Y_t^\sigma-Y_\infty^\sigma\right)&\geq-\EE\left(\left|Y_t^\sigma-Y_\infty^\sigma\right|\right)\\
&\geq-\mathbb{W}_2\left(\mathcal{L}\left(Y_t^\sigma\right);\mu^{m(\sigma)-\delta,\sigma}\right)\\
&\geq-e^{-(\alpha-\theta)t}\mathbb{W}_2\left(\mathcal{L}\left(X_0\right);\mu^{m(\sigma)-\delta,\sigma}\right)\\
&\geq-\mathbb{W}_2\left(\mathcal{L}\left(X_0\right);\mu^{m(\sigma)-\delta,\sigma}\right)\,.
\end{align*}

Hence, for any $t\leq T_0^\sigma$, we have:
\begin{equation*}
\EE(Y_t^\sigma)\geq M\left(\mu^{m(\sigma)-\delta,\sigma}\right)-\mathbb{W}_2\left(\mathcal{L}\left(X_0\right);\mu^{m(\sigma)-\delta,\sigma}\right)\,.
\end{equation*}
We deduce:

\begin{equation*}
\EE(Y_t^\sigma)\geq\left(m(\sigma)-\delta\right)+\chi_\sigma\left(m(\sigma)-\delta\right)-\mathbb{W}_2\left(\mathcal{L}\left(X_0\right);\mu^{m(\sigma)-\delta,\sigma}\right)\,.
\end{equation*}

Consequently, for any $t\leq T_0^\sigma$, $\EE(X_t^\sigma)\geq \left(m(\sigma)-\delta\right)+\chi_\sigma\left(m(\sigma)-\delta\right)-\mathbb{W}_2\left(\mathcal{L}\left(X_0\right);\mu^{m(\sigma)-\delta,\sigma}\right)>m(\sigma)-\delta$. Hence, $T_0^\sigma=+\infty$ so that for any $t\geq0$: $\EE(X_t^\sigma)\geq m(\sigma)-\delta$.

By applying main theorem in \cite{AOP}, we deduce that $\mathcal{L}\left(X_t^\sigma\right)$ weakly converges towards one of the steady state: $\nu_-^\sigma$, $\nu_0^\sigma$ or $\nu_+^\sigma$.

Since $M\left(\nu_-^\sigma\right)<0=M\left(\nu_0^\sigma\right)<M\left(\mathcal{L}\left(X_t^\sigma\right)\right)$ for any $t\geq0$, we deduce that $\mathcal{L}\left(X_t^\sigma\right)$ does not converge towards either $\nu_-^\sigma$ nor $\nu_0^\sigma$. Consequently, $\mathcal{L}\left(X_t^\sigma\right)$ weakly converges towards $\nu_+^\sigma$ as $t$ goes to infinity.

In the other hand, if $\alpha\leq\theta$, we can use Poincaré inequality from \cite[Corollary 1.6.]{BBCG} since $V''(x)+\alpha>0$ for sufficiently large $x$. We deduce the existence of $C(\sigma)>0$ such that

\begin{equation*}
{\rm L}_2\left(\mathcal{L}\left(Y_t^\sigma\right);\mu^{m(\sigma)-\delta,\sigma}\right)\leq e^{-C(\sigma)t}{\rm L}_2\left(\mathcal{L}\left(X_0\right);\mu^{m(\sigma)-\delta,\sigma}\right)\,,
\end{equation*}

so that

\begin{equation*}
\mathbb{W}_2\left(\mathcal{L}\left(Y_t^\sigma\right);\mu^{m(\sigma)-\delta,\sigma}\right)\leq e^{-C(\sigma)t}{\rm L}_2\left(\mathcal{L}\left(X_0\right);\mu^{m(\sigma)-\delta,\sigma}\right)\,.
\end{equation*}

We then complete the proof as previously with ${\rm L}_2\left(\mathcal{L}\left(X_0\right);\mu^{m(\sigma)-\delta,\sigma}\right)$ instead of $\mathbb{W}_2\left(\mathcal{L}\left(X_0\right);\mu^{m(\sigma)-\delta,\sigma}\right)$.

\begin{small}
\def\cprime{$'$}

\end{small}

\end{document}